\documentclass[11pt]{amsart}
\usepackage{amsopn}
\usepackage{amssymb, amscd}
\usepackage{multirow}
\usepackage{hyperref}

\newtheorem{theorem}{Theorem}[section]
\newtheorem{lemma}[theorem]{Lemma}
\newtheorem{proposition}[theorem]{Proposition}

\theoremstyle{definition}
\newtheorem{definition}[theorem]{Definition}
\newtheorem{remark}[theorem]{Remark}
\numberwithin{equation}{section}

\newcommand{\nc}{\newcommand}
\nc{\ngo}{\mathfrak{n} }  \nc{\ggo}{\mathfrak{g} }  \nc{\sog}{\mathfrak{so} } \nc{\glg}{\mathfrak{gl} } \nc{\lca}{\mathcal{L}} \nc{\vp}{\varphi}
\nc{\dpar}{\frac{\partial}{\partial t}} \nc{\Gl}{\mathrm{GL}} \nc{\RR}{{\Bbb R}} \nc{\la}{\langle} \nc{\ra}{\rangle} \nc{\ip}{\langle\cdot,\cdot\rangle}
\nc{\lb}{[\cdot,\cdot]} \nc{\Diff}{\operatorname{Diff}} \nc{\vol}{\operatorname{vol}} \nc{\Der}{\operatorname{Der}} \nc{\Aut}{\operatorname{Aut}}
\nc{\End}{\operatorname{End}} \nc{\Ricci}{\operatorname{Ric}} \nc{\tr}{\operatorname{tr}} \nc{\Diag}{\operatorname{Diag}}

\author{Marina Nicolini}
\address{Universidad Nacional de C\'ordoba\\ FaMAF and CIEM\\ 5000 C\'ordoba\\ Argentina.} \email{mnicolini@famaf.unc.edu.ar}
\thanks{Partially supported by a CONICET doctoral fellowship, a Consejo Interuniversitario Nacional fellowship and a research grant from CONICET}

\begin{document}

\title[Laplacian solitons on nilpotent Lie groups]{Laplacian solitons on nilpotent Lie groups}

\begin{abstract}
We investigate the existence of closed $G_2$-structures which are solitons for the Laplacian flow on nilpotent Lie groups. We obtain that seven of the
twelve Lie algebras admitting a closed $G_2$-structure do admit a Laplacian soliton. Moreover, one of them admits a continuous family of Laplacian solitons
which are pairwise non-homothetic and the Laplacian flow evolution of four of them is not diagonal.
\end{abstract}

\maketitle

\section{Introduction}
A closed $G_2$-structure $\vp$ on a $7$-manifold $M$ is said to be a \emph{Laplacian soliton} if
\begin{equation}\label{sol}
\Delta_\vp \vp = c \vp + \lca_X \vp,
\end{equation}
for some $c\in\RR$ and vector field $X$ on $M$, where $\Delta_\vp$ is the Hodge Laplacian on forms defined by $\vp$ and $\lca_X$ denotes the Lie
derivative. Laplacian solitons are also characterized as the $G_2$-structures that evolves self-similarly  under the Laplacian flow $\dpar\vp(t) =
\Delta_{\vp(t)}\vp(t)$ introduced by Bryant in \cite{Bry} (see \cite{Lty} for further information).

For left-invariant $G_2$-structures on a simply connected Lie group $G$, one has the following `algebraic' versions of Laplacian solitons (see \cite{L2}):
a \emph{semi-algebraic soliton} is a Laplacian soliton for which the field $X$ is defined by the one-parameter subgroup of automorphisms of $G$ associated
to some derivation $D$ of the Lie algebra $\ggo$ of $G$. If $D^t$ is also a derivation, then it is called an \emph{algebraic soliton}, which is known to be
equivalent to evolve `diagonally' under the Laplacian flow (see \cite[Theorem $4.10$]{L2}).

Conti and Fern\'andez proved in \cite{CF} that there are, up to isomorphism, twelve $7$-dimen\-sion\-al nilpotent Lie algebras that admit a left-invariant
closed $G_2$-structure. On the other hand, Fern\'andez, Fino and Manero studied in \cite{FFM} the existence of left-invariant closed $G_2$-structures
defining a Ricci soliton metric among the Lie algebras given in \cite{CF}. It is also natural to ask which of these twelve Lie algebras admit a closed
Laplacian soliton. In this paper, we find a closed Laplacian soliton on each of the first seven Lie algebras. Our main result is summarized as follows.

\begin{theorem}\label{teo} For each $i=1,\dots,7$, let $\ngo_i$ be the Lie algebra given in Table \ref{nl}.
\begin{itemize}
\item[(i)] $\ngo_1$ and $\ngo_2$ admit an algebraic soliton (see Table \ref{n2-n3}).
\item[(ii)] $\ngo_3$ admits a pairwise non-homothetic one-parameter family of algebraic solitons (see Table \ref{n2-n3}).
\item[(iii)] Each of $\ngo_4, \ngo_5, \ngo_6, \ngo_7$ does admit a semi-algebraic soliton which is not algebraic (see Table \ref{n4-n5} and Table \ref{n6-n7}).
\end{itemize}
\end{theorem}

The Laplacian solitons obtained are all expanding (i.e.\ $c > 0$ in \eqref{sol}). It is not hard to see that in the cases $\ngo_1$ and $\ngo_2$, the
Laplacian soliton is also a Ricci soliton. In cases $\ngo_4$ and $\ngo_6$, the Laplacian soliton we found is not a Ricci soliton, though $\ngo_4$ and
$\ngo_6$ are known to admit closed $G_2$-structures with Ricci soliton associated metrics. It is also proved in \cite{FFM} that the remaining algebras $\ngo_3$, $\ngo_5$ and  $\ngo_7$ do not
admit a closed $G_2$-structure with Ricci soliton associated metric.

The family of non-homothetic Laplacian solitons found on $\ngo_3$ shows that the uniqueness up to isometry and scaling of Ricci solitons on nilpotent Lie
algebras (see \cite{L1}) does not hold in the Laplacian case.  Another relevant difference between Laplacian and Ricci solitons is the fact that any
homogeneous Ricci soliton is isometric to an algebraic soliton (see \cite{J}).  On the contrary, we proved that four of the Lie algebras admit
semi-algebraic Laplacian solitons that are not equivalent to any algebraic soliton.

It would be desirable to find a Laplacian soliton on every Lie algebra in Table \ref{nl}, but with our method the computations became too complicated. Indeed, the Ricci
soliton on $\ngo_{10}$, whose existence was proved in  \cite[Example $2$]{FC}, is not known explicitly; moreover, the existence of a closed $G_2$-structure with a
Ricci soliton associated metric on $\ngo_{10}$ is still open (see \cite[Remark $3.5$]{FFM}).

\section{Preliminaries}

Given a $7$-dimensional differentiable manifold $M$, we consider a differentiable $3$-form $\vp\in\Omega^3 M$. For each $p\in M$, $\vp_p$ is said to be
\emph{positive} if there exists a basis $\{e_1,\dots,e_7\}$ of $T_pM$ such that
\begin{equation}\label{phi}
\vp_p= e^{127}+e^{347}+e^{567}+e^{135}-e^{146}-e^{236}-e^{245},
\end{equation}
where $e^{ijk}:=e^{i}\wedge e^{j}\wedge e^{k}$ and $\{e^1,\dots,e^7\}$ is the dual basis of  $\{e_1,\dots,e_7\}$. When $\vp_p$ is positive for every $p\in
M$, we call $\vp$ a \emph{$G_2$-structure}. Any $G_2$-structure induces a Riemannian metric $g_\vp$ and an orientation, and so a Hodge star operator
denoted by $\ast_\vp:\Omega M \rightarrow \Omega M$. The Hodge star operator in combination with the differential of forms on $M$ define the Hodge
Laplacian operator $\Delta_\vp$. In particular, on $3$-forms, $\Delta_\vp : \Omega^3 M \rightarrow \Omega^3 M$ is given by $\Delta_\vp=\ast_\vp d \ast_\vp
d-d \ast_\vp d \ast_\vp$.

For a one-parameter family  $\vp(t)$ of $G_2$-structures on $M$, we have a natural geometric flow, introduced by R. Bryant in $1992$, given by
\begin{equation}\label{intro-FL}
\dpar\vp(t) = \Delta_{\vp(t)}\vp(t),
\end{equation}
so called the \emph{Laplacian flow} (see \cite{Bry}). A $G_2$-structure $\vp$ on a $7$-differentiable manifold flows in a self-similar way along the
Laplacian flow, i.e.\ the solution $\vp(t)$ with $\vp(0)=\vp$ has the form
$$
\vp(t)=c(t)f(t)^*\vp, \quad\mbox{ for some } c(t)\in\RR^* \mbox{ and } f(t)\in\Diff(M),
$$
if and only if
$$
\Delta_\vp\vp=c\vp+\lca_{X}\vp, \qquad \mbox{for some}\quad c\in\RR, \quad X\in\mathfrak{X}(M)\; \mbox{(complete)},
$$
where $\lca_X$ denotes the Lie derivative. In that case, $c(t)=\left(\frac{2}{3}ct+1\right)^{3/2}$ and $\vp$ is called a \emph{Laplacian soliton}.
Furthermore, $\vp$ is said to be \emph{expanding}, \emph{steady} or \emph{shrinking}, when $c>0,$ $c=0$ or $c<0$, respectively.

A $G_2$-structure $\vp$ on a $7$-differentiable manifold is said to be \emph{closed} if $d\vp=0$. In the closed case, the intrinsic torsion is only given
by the $2$-form
$$
\tau_\vp = - \ast_\vp  d \ast_\vp \vp, \qquad d\tau_\vp = \Delta_\vp \vp.
$$

We now consider a $7$-dimensional vector space $\ggo$. It is known that a $3$-form $\psi\in\Lambda^3\ggo^*$ is positive, i.e.\ $\psi$ can be written as
\begin{equation}
\vp_0:= e^{127}+e^{347}+e^{567}+e^{135}-e^{146}-e^{236}-e^{245},
\end{equation}
relative to some basis $\{e_1,\dots,e_7\}$ of $\ggo$, if and only if $\psi$ is in the orbit $\Gl(\ggo)\cdot\vp_0$. Here the action is given by,
\begin{equation}\label{Gl-action}
(h\cdot\phi)(X_1,\dots,X_k)=\phi(h^{-1}X_1,\dots,h^{-1}X_k), \qquad \forall X_1,\dots,X_k \in\ggo, \quad \phi\in\Lambda^k\ggo^*.
\end{equation}
Also, we know that $\vp_0$ induces an inner product on $\ggo$ as follows:
$$
\la X , Y \ra_{\vp_0} \vol_0 := \tfrac{1}{6} \iota_X \vp_0 \wedge \iota_Y \vp_0 \wedge \vp_0,
$$
where $\vol_0:=e^{1\dots7}$ and $\iota_X$ is defined by $(\iota_X \phi)(\cdot,\cdot):=\phi(X,\cdot,\cdot)$. It is easy to check that the basis
$\{e_1,\dots,e_7 \}$ is orthonormal with respect to the inner product $\ip_{\vp_0}$ and oriented relative to $\vol_0$.

Every positive $3$-form $\psi=h\cdot\vp_0$ with $h\in\Gl(\ggo)$ defines an inner product $\ip_\psi$ and a volume form $\vol_\psi$ by
\begin{equation}\label{pi-phi}
\ip_\psi:=\la h^{-1}\cdot,h^{-1}\cdot \ra_{\vp_0}, \qquad \vol_\psi := h\cdot\vol_0.
\end{equation}

If $\{f_1,\dots,f_7\}$ is an orthonormal basis of $(\ggo,\ip_\psi)$, then we also denote by $\ip_\psi$ the inner product on $\Lambda^k\ggo^*$, which makes
of $\{f^{i_1\dots i_k} : i_1<\dots<i_k \}$ an orthonormal basis.

The following facts are direct consequences of the above definitions.

\begin{lemma} \label{lema1} Let $\ggo$ be a $7$-dimensional vector space. If $X,Y\in\ggo$, $h\in\Gl(\ggo)$ and
$\psi\in\Lambda^3\ggo^*$ is positive, then,
\begin{itemize}

\item[(i)] $ \la X , Y \ra_\psi  \vol_\psi =\tfrac{1}{6} \iota_X \psi \wedge
\iota_Y \psi \wedge\psi$.

\item[(ii)] $\la X,Y \ra_{h\cdot \psi}=\la h^{-1} X,h^{-1}Y \ra_\psi, \quad \forall
X,Y\in\ggo, \quad ($i.e.\ $\ip_{h\cdot\psi}=h\cdot\ip_\psi)$.

\item[(iii)] $ \ip_{c \psi}=c^{\frac{2}{3}} \ip_\psi, \quad \forall c \in\RR^*$.

\end{itemize}\end{lemma}

For our next lemma, we need to introduce a definition. Let $\ggo$ be a Lie algebra and $G$ the corresponding simply connected Lie group. We note that each
positive $3$-form $\vp\in\Lambda^3\ggo^*$ defines a left-invariant $G_2$-structure on $G$. Given $D\in\Der(\ggo)$ and $t\in\RR$, we denote by
$f_t\in\Aut(G)$ the automorphism such that $df_t|_e=e^{tD}\in\Aut(\ggo)$ and by $X_D$ the corresponding vector field on $G$:
$$
X_D(a):= \frac{d}{dt}\Big{|}_0 f_t(a), \quad \forall a\in G.
$$

It is easy to prove that the Lie derivative of a left-invariant form $\psi\in\Lambda^k\ggo^*$ with respect to $X_D$ is given by
\begin{equation}\label{prop-lieder}
(\lca_{X_D}\psi)(X_1,\dots, X_k):=\psi(DX_1,X_2,\dots,X_k)+ \dots + \psi(X_1,X_2,\dots,DX_k),
\end{equation}
for all $X_1,\dots,X_k \in \ggo$.

\begin{lemma}\label{inv} Let $\ggo$ be a $7$-dimensional Lie algebra and consider
$\psi\in\Lambda^k\ggo^*$, $h\in\Aut(\ggo)$.
\begin{itemize}
\item[(i)]$ d(h\cdot\psi)=h\cdot d\psi.$
\item[(ii)]If $k=3$ and $\psi$ is positive, then
    \begin{itemize}
    \item[(a)] $\Delta_{h\cdot\psi} h\cdot\psi=h\cdot\Delta_{\psi}\psi.$
    \item[(b)]$\Delta_{c\psi} c\psi=c^{\frac{1}{3}}\Delta_{\psi} \psi, \quad \forall c\in\RR^*$.
    \end{itemize}
\item[(iii)]$\lca_{X_{hDh^{-1}}}(h\cdot\psi)=h\cdot\lca_{X_D}\psi$, for any $D\in\Der(\ggo)$.
\end{itemize}
\end{lemma}

\begin{proof}
We first prove part (i). For $X_1,X_2, \dots, X_{k+1} \in \ggo$, one has that
\begin{align*}
d(&h \cdot  \psi)(X_1,\dots, X_{k+1})=\\
                    =&\sum_{i<j}(-1)^{i+j} (h\cdot\psi )([X_i,X_j],X_1,\dots,\widehat{X_i},\dots,\widehat{X_j},\dots,X_{k+1})\\
                               =&\sum_{i<j}(-1)^{i+j} \psi (h^{-1}[X_i,X_j],h^{-1}X_1,\dots,\widehat{h^{-1}X_i},\dots,\widehat{h^{-1}X_j},\dots,h^{-1}X_{k+1})\\
                               =&\sum_{i<j}(-1)^{i+j} \psi ([h^{-1}X_i,h^{-1}X_j],h^{-1}X_1,\dots,\widehat{h^{-1}X_i},\dots,\widehat{h^{-1}X_j},\dots,h^{-1}X_{k+1})\\
                               = & d\psi(h^{-1}X_1,\dots, h^{-1}X_{k+1})= (h\cdot d\psi) (X_1,\dots, X_{k+1}).
\end{align*}
To prove part (ii), let us first show that $\ast_{h\cdot\psi}h\cdot\phi=h\cdot\ast_{\psi}\phi$, for any $\phi\in\Lambda^s\ggo^*$. To do this, fix
$\alpha\in\Lambda^{s}\ggo^*$ and $X,Y \in \ggo$. Thus
\begin{align*}
\alpha \wedge \ast_{h\cdot\psi} h\cdot \phi =& \la \alpha , h \cdot\phi \ra_{h \cdot \psi} \vol_{h \cdot \psi}= \la h^{-1} \cdot \alpha , h^{-1} \cdot( h \cdot\phi )\ra_{\psi} h\cdot \vol_{\psi}\\
                                            =& h\cdot ( \la h^{-1} \cdot \alpha , \phi \ra_{\psi}\vol_{\psi} )= h\cdot ( h^{-1} \cdot \alpha \wedge \ast_\psi \phi )\\
                                            =& \alpha \wedge h\cdot ( \ast_\psi \phi).
\end{align*}
On the other hand, it is easy to see that $\vol_{h \cdot \psi}=h\cdot \vol_{\psi}$ and $\la\alpha,\beta\ra_{h\cdot\psi}=\la h^{-1} \cdot \alpha,h^{-1}
\cdot \beta\ra_\psi, $ for every $\alpha, \beta\in\Lambda^s\ggo^*$.

Now we can prove (a) as follows:
\begin{align*}
- \Delta_{h\cdot\psi} (h\cdot\psi) =&d\ast_{h\cdot\psi}d\ast_{h\cdot\psi}(h\cdot\psi) -\ast_{h\cdot\psi}  d\ast_{h\cdot\psi}  d(h\cdot\psi)\\
                                          =&d\ast_{h\cdot\psi}d(h\cdot (\ast_{\psi}\psi)) -\ast_{h\cdot\psi}  d\ast_{h\cdot\psi}(h\cdot(d\psi))\\
                                          =&d\ast_{h\cdot\psi}(h\cdot(d\ast_{\psi}\psi)) -\ast_{h\cdot\psi}  d(h\cdot (\ast_{\psi}d\psi))\\
                                          =&d (h\cdot (\ast_{\psi}d\ast_{\psi}\psi)) -\ast_{h\cdot\psi}(h\cdot (d\ast_{\psi}d\psi))\\
                                          =& h\cdot (d\ast_{\psi}d\ast_{\psi}\psi) -h\cdot (\ast_{\psi}d\ast_{\psi}d\psi)\\
                                          =& h\cdot (d\ast_{\psi}d\ast_{\psi}\psi -\ast_{\psi}d\ast_{\psi}d\psi)\\
                                          =&-h\cdot\Delta_{\psi}\psi.
\end{align*}
For item (b), we first observe that $\ast_{c\psi}\phi=c^{\frac{7-2s}{3}}\ast_{\psi}  \phi $, for all $c\in\RR^*$, $\phi\in\Lambda^s\ggo^*$. Indeed, since
$$
(c^{-\frac{1}{k}} I)\cdot\vp=c \vp, \quad \forall c\in\RR^*,\quad \vp\in\Lambda^k\ggo^*,
$$
$h\cdot\ast_\psi \phi=\ast_{h\cdot\psi}h\cdot \phi $, we obtain
\begin{align*}
c^{\frac{s}{3}}\ast_{c\psi} \phi&=\ast_{c\psi} c^{\frac{s}{3}}\phi=\ast_{(c^{-\frac{1}{3}}I)\cdot\psi} ((c^{\frac{s}{3}})^{-\frac{1}{s}}I)\cdot\phi=\ast_{(c^{-\frac{1}{3}}I)\cdot\psi}(c^{-\frac{1}{3}}I)\cdot\phi\\
&=(c^{-\frac{1}{3}}I)\cdot \ast_{\psi} \phi=c^{\frac{7-s}{3}}\ast_{\psi} \phi.
\end{align*}
We now proceed to show (b):
\begin{align*}
\Delta_{c\psi} c\psi&= - d \ast_{c\psi} d \ast_{c\psi} c\psi + \ast_{c\psi} d \ast_{c\psi} d \hspace{0.1 cm} c\psi=  c (- d \ast_{c\psi} d \ast_{c\psi} \psi +  \ast_{c\psi} d \ast_{c\psi} d \psi)\\
                    &=  c(- d\hspace{0.1 cm} c^{\frac{7-10}{3}} \ast_{\psi} d\hspace{0.1 cm} c^{\frac{7-6}{3}} \ast_{\psi} \psi +  c^{\frac{7-8}{3}}\ast_{\psi} d \hspace{0.1 cm}c^{\frac{7-8}{3}}\ast_{\psi} d \psi)\\
                    &=  c\hspace{0.1 cm}c^{-\frac{2}{3}}(- d  \ast_{\psi} d  \ast_{\psi} \psi +  \ast_{\psi} d \ast_{\psi} d \psi)=  c^{\frac{1}{3}}\Delta_{\psi} \psi.
\end{align*}
Finally, part (iii) follows from \eqref{prop-lieder}: for any $X_1,\dots,X_k\in\ggo$,
\begin{align*}
\lca_{X_{hDh^{-1}}}(&h\cdot \psi)(X_1,\dots,X_k)=\\
            =&(h\cdot\psi)(hDh^{-1}X_1,\dots,X_k)+\dots+(h\cdot \psi)(X_1,\dots,hDh^{-1}X_k)\\
            =&\psi(Dh^{-1}X_1,\dots,h^{-1}X_k)+\dots+\psi(h^{-1}X_1,\dots,Dh^{-1}X_k)\\
            =&\lca_{X_D}\psi(h^{-1}X_1,\dots,h^{-1}X_k)=(h\cdot\lca_{X_D}\psi)(X_1,\dots,X_k),
\end{align*}
and this concludes the proof of the lemma.
\end{proof}

Laplacian solitons on Lie groups have been deeply studied in \cite{L2}. The following definition will be used from now on along the paper.

\begin{definition}
Given $\ggo$ a $7$-dimensional Lie algebra and $\psi$ a positive $3$-form on $\ggo$, we call $(\ggo,\psi)$ a \emph{semi-algebraic soliton} if there exist
$D\in\Der (\ggo)$ and $\lambda\in\RR$ such that
\begin{equation}
\Delta_\psi \psi=\lca_{X_D}\psi+\lambda \psi.
\end{equation}
In the case when $D^t\in\Der(\ggo)$, we say that $(\ggo,\psi)$ is an \emph{algebraic soliton}.
\end{definition}

Let $\theta:\glg(\ggo)\rightarrow \End(\Lambda^3\ggo^*)$ be the derivative of the action given by \eqref{Gl-action}, i.e.\
$$
\theta(A)\psi(\cdot,\cdot,\cdot)=-\psi(A\cdot,\cdot,\cdot)-\psi(\cdot,A\cdot,\cdot)-\psi(\cdot,\cdot,A\cdot), \qquad \forall A\in\glg(\ggo), \quad
\psi\in\Lambda^3\ggo^*.
$$
It is shown in \cite[(11)]{L2} that for any closed $G_2$-structure $\psi$ on $\ggo$, there exists a unique symmetric operator $Q_\psi\in\glg(\ggo)$ such
that $\theta(Q_\psi)\psi=\Delta_\psi\psi$. The following useful formula for $Q_\psi$ was given in \cite[Proposition $2.2$]{L2}: for any closed
$G_2$-structure $\psi$,
\begin{equation}\label{lot-wei}
Q_\psi=\Ricci_\psi - \frac{1}{12} \tr (\tau_\psi^2) I + \frac{1}{2} \tau_\psi^2,
\end{equation}
where $\Ricci_\psi$ is the Ricci operator of $(G,g_\psi)$ and $\tau_\psi\in\sog(TG)$ also denotes the skew-symmetric operator determined by the $2$-form
$\tau_\psi$ (i.e.\ $\tau_\psi=\la \tau_\psi\cdot,\cdot\ra_\psi$).

According to \cite[Proposition $4.5$]{L2}, $(\ggo,\psi)$ is a semi-algebraic soliton with $\Delta_\psi \psi=\lca_{X_D}\psi+\lambda \psi$, if and only if
$Q_\psi=-\frac{1}{3}\lambda I - \frac{D+D^t}{2}$. Recall that $\psi$ is an algebraic soliton if and only if $\frac{D+D^t}{2}\in\Der(\ggo)$.

\begin{definition}\label{def-equiv} We say that two $G_2$-structures $(\ggo_1,\psi_1)$ and $(\ggo_2,\psi_2)$ are \emph{equivalent}
if there exists a Lie algebra isomorphism $h:\ggo_1 \rightarrow \ggo_2$ such that $h\cdot\psi_1=\psi_2$. We denote it briefly by
$(\ggo_1,\psi_1)\simeq(\ggo_2,\psi_2)$. Also, we say that $(\ggo_1,\psi_1)$ and $(\ggo_2,\psi_2)$ are \emph{homothetic} if there exists $c\in\RR^*$ such
that $(\ggo_1,\psi_1)\simeq(\ggo_2, c \psi_2)$.
\end{definition}

\begin{proposition} Let $\ggo$ be a $7$-dimensional Lie algebra, $\psi_1,\psi_2\in\Lambda^2\ggo^*$ positive such that
$(\ggo,\psi_1)$ and $(\ggo,\psi_2)$ are homothetic. Then $(\ggo,\psi_1)$ is a semi-algebraic soliton if and only if $(\ggo,\psi_2)$ is so.
\end{proposition}

\begin{proof}
Recall that $(\ggo,\psi_1)$ is semi-algebraic soliton if and only if there exist $D\in\Der(\ggo)$ and $\lambda\in\RR$ such that $\Delta_{\psi_1} \psi_1 =
\lca_{X_D}\psi_1 + \lambda \psi_1$. Therefore, by Lemma \ref{inv}, we have that
\begin{align*}
c^{\frac{1}{3}}\Delta_{ \psi_2 }\psi_2&=\Delta_{c \psi_2} (c \psi_2)=\Delta_{h\cdot \psi_1} (h\cdot \psi_1)=h\cdot \Delta_{\psi_1} \psi_1=h\cdot (\lca_{X_D}\psi_1 + \lambda \psi_1)\\
&= \lca_{X_{h D h^{-1}}}(h\cdot \psi_1) + \lambda (h\cdot \psi_1)=\lca_{X_{h D h^{-1}}}(c \psi_2) + \lambda (c \psi_2)=c\lca_{X_{h D h^{-1}}} \psi_2 + c
\lambda  \psi_2.
\end{align*}
So, $\Delta_{\psi_2 }\psi_2=c^{\frac{2}{3}}\lca_{X_{h D h^{-1}}} \psi_2 + c^{\frac{2}{3}} \lambda  \psi_2=\lca_{X_{c^{\frac{2}{3}}h D h^{-1}}} \psi_2 +
c^{\frac{2}{3}} \lambda  \psi_2$. Since $c^{\frac{2}{3}}h D h^{-1}\in\Der(\ggo)$, we conclude that $(\ggo,  \psi_2)$ is a semi-algebraic soliton.
\end{proof}

\section{Closed Laplacian solitons}

In \cite{CF}, Conti and Fern\'andez studied the existence of closed $G_2$-structures on a $7$-dimensional nilpotent Lie algebra. They obtained that, up to
isomorphism, there are $12$ nilpotent Lie algebras with that property, which are shown in Table \ref{nl}. It is of interest to know whether these Lie
algebras admit closed Laplacian solitons.

{\small \begin{table}
\renewcommand{\arraystretch}{1.6}
$$
\begin{array}{|c|c|}\hline
  \ggo & \mbox{Lie bracket} \\
\hline\hline
\ngo_1 & \lb = 0 \\
\hline
\ngo_2 & [e_1,e_2]=-e_5, [e_1,e_3]=-e_6  \\
\hline
\ngo_3 & [e_1,e_2]=-e_4, [e_1,e_3]=-e_5, [e_2,e_3]=-e_6\\
\hline
\ngo_4 & [e_1,e_2]=-e_3, [e_1,e_3]=-e_6, [e_2,e_4]=-e_6, [e_1,e_5]=-e_7\\
\hline
\ngo_5 & [e_1,e_2]=-e_3,[e_1,e_3]=-e_6, [e_1,e_4]=-e_7, [e_2,e_5]=-e_7 \\
\hline
\ngo_6 & [e_1,e_2]=-e_4, [e_1,e_3]=-e_5, [e_1,e_4]=-e_6, [e_1,e_5]=-e_7\\
\hline
\ngo_7 & [e_1,e_2]=-e_4, [e_1,e_3]=-e_5, [e_1,e_4]=-e_6, [e_2,e_3]=-e_6, [e_1,e_5]=-e_7  \\
\hline
\ngo_8 & \begin{array}{c}[e_1,e_2]=-e_3, [e_1,e_3]=-e_4, [e_2,e_3]=-e_5, [e_1,e_5]=-e_6,\;\;\ \\
         \;\;\  [e_2,e_4]=-e_6, [e_1,e_6]=-e_7,[e_3,e_4]=-e_7 \end{array} \\
\hline
\ngo_9 & \begin{array}{c}[e_1,e_2]=-e_3, [e_1,e_3]=-e_4, [e_2,e_3]=-e_5, [e_1,e_5]=-e_6, \;\;\ \\
        \;\;\ [e_2,e_4]=-e_6, [e_1,e_6]=-e_7, [e_3,e_4]=-e_7, [e_2,e_5]=-e_7 \end{array}\\
\hline
\ngo_{10} & \begin{array}{c}[e_1,e_2]=-e_3, [e_1,e_3]=-e_5, [e_2,e_4]=-e_5, [e_1,e_4]=-e_6, \;\;\ \\
         \;\;\ [e_4,e_6]=-e_7, [e_3,e_4]=-e_7, [e_1,e_5]=-e_7, [e_2,e_3]=-e_7 \end{array} \\
\hline
\ngo_{11} & \begin{array}{c}[e_1,e_2]=-e_3, [e_1,e_3]=-e_5, [e_2,e_4]=-e_6, [e_2,e_3]=-e_6,  \;\;\ \\
         \;\;\ [e_2,e_5]=-e_7, [e_3,e_4]=-e_7, [e_1,e_5]=-e_7, [e_1,e_6]=-e_7, [e_2,e_6]=3e_7 \end{array}\\
\hline
\ngo_{12} & \begin{array}{c}[e_1,e_2]=-e_4, [e_2,e_3]=-e_5, [e_1,e_3]=e_6, [e_2,e_6]=-2e_7,  \;\;\ \\
        \;\;\ [e_3,e_4]=2e_7, [e_1,e_6]=2e_7, [e_2,e_5]=-2e_7\end{array}\\
\hline
\end{array}
$$
\caption{Nilpotent Lie algebras that admit a closed $G_2$-structure (see \cite{CF}).}\label{nl}
\end{table}}

We prove that for the first seven Lie algebras of the table, there exists at least one closed Laplacian soliton.

\begin{theorem} For each $i=1,\dots,7$, let $\ngo_i$ be the Lie algebra given in Table \ref{nl}.
\begin{itemize}
\item[(i)] $\ngo_2$ admits an algebraic soliton (see Table \ref{n2-n3}).
\item[(ii)] $\ngo_3$ admits a pairwise non-homothetic one-parameter family of algebraic solitons (see Table \ref{n2-n3}).
\item[(iii)] Each of $\ngo_4, \ngo_5, \ngo_6, \ngo_7$ does admit a semi-algebraic soliton which is not algebraic (see Table \ref{n4-n5} and Table \ref{n6-n7}).
\end{itemize}
\end{theorem}

\begin{proof}
We only give a proof for the cases $\ngo_3$ and $\ngo_4$, the other cases follow in much the same way.

To prove that $\ngo_3$ admits a family of algebraic solitons up to isomorphism and scaling, we consider $\ngo_3(a,b,c)$ to be the $7$-dimensional Lie
algebra with basis $\{e_1,\dots,e_7\}$ and Lie bracket defined by
$$[e_1,e_2]=-ae_4, \quad [e_1,e_3]=-be_5,\quad  [e_2,e_3]=-ce_6, \quad a,b,c\in\RR^*,$$
or equivalently,
\begin{equation}\label{corchete3}
de^{12}=ae^4, \quad de^{13}=be^5, \quad de^{23}=ce^6, \qquad a,b,c\in\RR^*.
\end{equation}
We have a linear isomorphism that carries $\ngo_3(1,1,1)$ into $\ngo_3(a,b,c)$, whose matrix is $\Diag(1,1,a,ab/c,1,ab,d)$. From now on $\ngo_3$ denotes
$\ngo_3(a,b,c)$. We consider the $3$-form
$$
\vp_3=e^{123}+e^{145}+e^{167}+e^{246}-e^{257}-e^{347}-e^{356}\in\Lambda^3\ngo_3^*.
$$
If $h_3\in\Gl_7(\RR)$ is the permutation $(1,6,4,3,5,2,7)$, then $ h_3 \cdot \vp_3 = \vp_0$, which implies that $\vp_3$ is positive. It is easy to check by
using \eqref{corchete3} that $d\vp_3=(a-b-c)e^{1237}$, so $\vp_3$ is closed if and only if $a=b+c$. If we assume $\vp_3$ to be closed, then the Laplacian
can be computed as follows:
\begin{align*}
\ast\vp_3 &= -e^{1247}-e^{1256}-e^{1346}+e^{1357}+e^{2345}+e^{2367}+e^{4567},\\
d\ast\vp_3 &= a e^{12567}-be^{13467}+ce^{23457},\\
\ast  d\ast \vp_3 &= ce^{16}-be^{25}+ae^{34},\\
d\ast  d\ast\vp_3 &= -(a^2+b^2+c^2) e^{123}.
\end{align*}
By replacing in the condition $a=b+c$, we obtain $\Delta_{\vp_3}\vp_3=2(b^2+c^2+bc)e^{123}$.

What is left to show is that $\Delta_{\vp_3}\vp_3 =  \lca_{X_D}\vp_3+\lambda\vp_3 $ for some $D\in\Der(\ngo_3)$ and $\lambda\in\RR$. We propose $ D :=
d\Diag(1,1,1,2,2,2,2)$ with $d\in\RR^*$, so the resulting Lie derivative of $\vp_3$ with respect to the field $X_D$ is
$$
\lca_{X_D}\vp_3=3d e^{123}+ 5 d e^{145}+ 5 d e^{167}+ 5 d e^{246}- 5 d e^{257}- 5 d e^{347}- 5 d e^{356}.
$$
It follows that $\Delta_{\vp_3}\vp_3 =  \lca_{X_D}\vp_3+\lambda\vp_3 $ if and only if $\lambda=- 5 d $ and $d=-(b^2+c^2+bc)$. Since $D=D^t$, one obtains
that $(\ngo_3,\vp_3)$ is an algebraic soliton.

\begin{lemma} If $a,b,c\in\RR^*$ and $\ngo_3(a,b,c)$
are as above, then
\begin{itemize}
\item[(i)] $\vp_3$ is closed if and only if $a=b+c$.
\item[(ii)] $(\ngo_3(b+c,b,c),\vp_3)$ is an algebraic soliton.
\end{itemize}
\end{lemma}

\begin{remark}
For all $b,c\in\RR^*$, the algebraic soliton $(\ngo_3(b+c,b,c),\vp_3)$ is expanding since $\lambda>0$.
\end{remark}

As we have two free parameters, it is natural to ask whether  there are two non-equivalent algebraic solitons on $\ngo_3$.

\begin{proposition}
There exists a pairwise non-homothetic continuous family of algebraic solitons on $\ngo_3$.
\end{proposition}

\begin{remark}
This is in contrast to the known uniqueness up to isometry and scaling of Ricci solitons on nilpotent Lie algebras (see \cite{L1}).
\end{remark}

\begin{proof}
By using \emph{e.g.} the formula for the Ricci operator given in \cite[(8)]{L1}, it is easy to see that
$$
\Ricci_{b,c}=\frac{1}{2}\Diag(-a^2-b^2,-a^2-c^2,-b^2-c^2,a^2,b^2,c^2,0),
$$
where $a=b+c$. Clearly, $\Ricci_{b,c}$ has three positives eigenvalues, one equal to zero and three negatives for each $b,c\in\RR^*$. If we set  $b:=1-t$
and $c:=t$ with $t\in(0,\frac{1}{2})$, then for every $t\in (0,\frac{1}{2})$ the positive eigenvalues are ordered in the following way:
$$
\frac{t^2}{2}  < \frac {1-2t+t^2}{2} <  \frac{1}{2}.
$$
Now, if $(\ngo_3(b_1+c_1,b_1,c_1),\vp_3)$ and $(\ngo_3(kb_2+kc_2,kb_2,kc_2),\vp_3)$ are equivalent for some $k\in\RR^*$ (where $b_i=1-t_i$, $c_i=t_i$),
then there are in particular isometric, hence
$$
\frac{1}{2}=k^2 \frac{1}{2}, \quad \frac {1-2t_1+t_1^2}{2}=k^2 \frac {1-2t_2+t_2^2}{2}, \quad \frac{t_1^2}{2}=k^2 \frac{t_2^2}{2},
$$
which implies that $k^2=1$ and $t_1=t_2$; concluding the proof of the proposition.
\end{proof}

\begin{remark}
Let $R_\vp$ denote the scalar curvature of $\vp$, i.e.\ $R_\vp=\tr \Ricci_\vp$. The number $\frac{R_\vp^2}{|\Ricci_\vp|^2}$ is therefore an invariant up to
isometry or scaling. For $(\ngo_3(b+c,b,c),\vp_3)$,  $\frac{R_{b,c}^2}{|\Ricci_{b,c}|^2}=\frac{1}{2}$ for all $b,c\in\RR^*$, so it can not be used to prove
non-homothety.
\end{remark}

It follows from \eqref{lot-wei} that $Q_{\vp_3}=\frac{a^2+b^2+c^2}{6}\Diag(-2,-2,-2,1,1,1,1)$. Note that this coincides with $-\frac{1}{3}\lambda I-D$
above.

We can now proceed to the proof of part (iii) for $\ngo_4$. Let $\ngo_4=\ngo_4(a,b,c,d)$ be the $7$-dimensional nilpotent Lie algebra with basis
$\{e_1,\dots,e_7\}$ and Lie bracket given by
$$
[e_1,e_2]=-a e_3,\hspace{0.4cm}[e_1,e_3]=-b e_6,\hspace{0.4cm} [e_2,e_4]=-c e_6, \hspace{0.4cm}[e_1,e_5]=-d e_7 \qquad a, b, c, d \in\RR^* ,
$$
or equivalently,
$$
de^3=ae^{12},\hspace{0.4cm}de^6=be^{13}+ce^{24},\hspace{0.4cm} de^7=de^{15} \qquad a, b, c, d \in\RR^*.
$$
We have a linear isomorphism that carries $\ngo_4(1,1,1,1)$ into $\ngo_4(a,b,c,d)$, whose matrix is $\Diag(1,1,a,ab/c,1,ab,d)$. From now on $\ngo_4$
denotes $\ngo_4(a,b,c,d)$.

We consider the $3$-form
$$
\vp_4=-e^{124}-e^{456}+e^{347}+e^{135}+e^{167}+e^{257}-e^{236}\in\Lambda^3\ngo_4^*.
$$
Let $h_4\in\Gl_7(\RR)$ be the permutation $(1,-6,3,4,5,2,7)$, then $h_4 \cdot \vp_4 = \vp_0$, which implies that $\vp_4$ is positive.

\begin{lemma} \label{prop-n4} If $a,b,c,d\in\RR^*$ and $\ngo_4(a,b,c,d)$ is
as above, then
\begin{itemize}
\item[(i)] $\vp_4$ is closed if and only if $a=c$ and $b=d$.
\item[(ii)] If $a^2=2b^2$, then $(\ngo_4(a,b,a,b),\vp_3)$ is a semi-algebraic soliton.
\end{itemize}
\end{lemma}

\begin{proof}
It is easy to see that $d\vp_4=(a-c)e^{1247}+(d-b)e^{1345}$, so $\vp_4$ is closed if and only if $a=c$ and $b=d$. Assuming $\vp_4$ to be closed we proceed
to compute the Laplacian $\Delta_{\vp_4}\vp_4$:
\begin{align*}
\ast\vp_4 &= e^{3567}+e^{1237}+e^{1256}-e^{2467}+e^{2345}+e^{1346}+e^{1457},\\
d\ast\vp_4 &= ae^{12567}-ce^{23457}+be^{12347}+de^{12456},\\
\ast  d\ast \vp_4 &= ae^{34}-ce^{16}+be^{56}-de^{37},\\
d\ast d\ast\vp_4 &= (a^2+c^2)e^{124}-(b^2+d^2)e^{135}-bce^{245}-ade^{127}.
\end{align*}

Replacing in the condition $a=c$ and $b=d$, we obtain $\Delta_{\vp_4}\vp_4= -2a^2 e^{124}+2b^2 e^{135}+abe^{245}+abe^{127}$.

To prove that $(\ngo_4,\vp_4)$ is a semi-algebraic soliton, we have to find some $\lambda\in\RR$ and $D\in\Der(\ngo_4)$ such that
$\Delta_{\vp_4}\vp_4=\lambda\vp_4+\lca_{X_D}\vp_4$. We propose
$$
D:= \left[\begin{smallmatrix}
-b^2 & 0 & 0 & 0 & 0 & 0 & 0\\
0 & -2b^2 & 0 & 0 & 0 & 0 & 0 \\
0 & 0 & -3b^2 & 0 & 0 & 0 & 0 \\
0 & 0 & 0 & -2b^2 & 0 & 0 & 0 \\
-ab  & 0 & 0 & 0 & -3b^2 & 0 & 0\\
0 & 0 & 0 & 0 & 0 & -4b^2& 0\\
0 & 0 & 0 & -ab & 0 & 0 & -4b^2
\end{smallmatrix}\right]
$$
and $\lambda=9b^2$. Then the Lie derivative equals
\begin{align*}
\lca_{X_D}\vp_4 =
& 5 b^2 e^{124}+9b^2e^{456}-abe^{146}-9b^2e^{347}-7b^2e^{135}-9b^2e^{167}\\
&+ abe^{146}-9b^2e^{257}+abe^{127}+abe^{245}+9b^2e^{236}.
\end{align*}
The soliton equation holds if $a^2=2b^2$, i.e.\ if $a^2=2b^2$ then
$$
\lca_{X_D}\vp_4+9b^2\vp_4=-4b^2e^{124}+2b^2e^{135}+abe^{127}+abe^{245}=\Delta_{\vp_4}\vp_4.
$$
Note that $(\ngo_4(a,b,a,b),\vp_4)$ is not an algebraic soliton. Indeed, $D^t\not\in \Der(\ngo_4)$ since $[D^te_2,e_7]+[e_2,D^te_7]=-ab[e_2,e_4]=abce_6\ne
0=D^t[e_2,e_7]$.
\end{proof}

\begin{remark}
For every $a,b\in\RR^*$ such that $a^2=2b^2$, $(\ngo_4(a,b,a,b),\vp_4)$ is an expanding semi-algebraic soliton since $\lambda>0$.
\end{remark}

On the other hand, we are interested in computing $Q_{\vp_4}$. It is not hard to see that $ \Ricci_{\vp_4}=\Diag
\left(-\frac{a^2+b^2+d^2}{2},-\frac{a^2+c^2}{2},\frac{a^2-b^2}{2} ,-\frac{c^2}{2} ,-\frac{d^2}{2} ,\frac{b^2+c^2}{2} ,\frac{d^2}{2}\right)$ and
$\tau_{\vp_4}= -ae^{34}+ce^{16}-be^{56}+de^{37}$. It follows from \eqref{lot-wei} that
$$
Q_{\vp_4}=\left[\begin{smallmatrix}
-\frac{\alpha}{3}&0&0&0&\frac{bc}{2}&0&0\\
0&\frac{\alpha-3a^2-3c^2}{6}&0&0&0&0&0\\
0&0&\frac{\alpha-3b^2-3d^2}{6}&0&0&0&0\\
0&0&0&\frac{\alpha-3a^2-3c^2 }{6}&0&0&\frac{ad}{2}\\
\frac{bc}{2}&0&0&0&\frac{\alpha -3b^2-3d^2}{6}&0&0\\
0&0&0&0&0&\frac{\alpha}{6}&0\\
0&0&0&\frac{ad}{2}&0&0&\frac{\alpha}{6}\\
\end{smallmatrix}\right],
$$
where $\alpha=a^2+b^2+c^2+d^2$. Thus, $Q_{\vp_4}=-\frac{1}{3}\lambda I-D$, where $\lambda$ and $D$ are as above.

The remaining cases are analogous and the following lemmas provide information about them.

\begin{lemma}If $a,b\in\RR^*$, $\ngo_2(a,b)$ is the Lie algebra with Lie bracket $[e_1,e_2]=-ae_5$, $[e_1,e_3]=-b e_6$ and $\vp_2:= e^{147}+e^{267}+e^{357} +e^{123}+e^{156}+e^{245}-e^{346}$, then
\begin{itemize}
\item[(i)]  $\vp_2$ is closed if and only if $a=b$.
\item[(ii)] $(\ngo_2(a,a),\vp_2)$ is an algebraic soliton.
\end{itemize}
\end{lemma}

\begin{lemma} If $\ngo_5(a,b,c,d)$ is the Lie algebra with Lie bracket given by
$$
[e_1,e_2]=-ae_3, \quad [e_1,e_3]=-be_6, \quad [e_1,e_4]=-ce_7,\quad [e_2,e_5]=-de_7,
$$
where $a,b,c,d\in\RR^*$ and $\vp_5:=e^{134}+e^{457}-e^{246}-e^{125}-e^{356}+e^{167}-e^{237}$, then
\begin{itemize}
\item[(i)] $\vp_5$ is closed if and only if $a=d$ and $b=c$.
\item[(ii)] If $a^2=2 b^2$, then $(\ngo_5(a,b,b,a),\vp_5)$ is a semi-algebraic soliton.
\end{itemize}
\end{lemma}

\begin{lemma}If $a,b,c,d\in\RR^*$, $\ngo_6(a,b,c,d)$ is the Lie algebra with Lie bracket given by
$$
[e_1,e_2]=-ae_4,\quad  [e_1,e_3]=-be_5, \quad [e_1,e_4]=-ce_6,\quad  [e_1,e_5]=-de_7
$$
and $\vp_6:=e^{123}+e^{347}+e^{356}+ e^{145}-e^{246}+e^{167}+e^{257}$, then
\begin{itemize}
\item[(i)] $\vp_6$ is closed if and only if $a=b$ and $c=d$.
\item[(ii)] If $a^2=2 c^2$, then  $(\ngo_6(a,a,c,c),\vp_6)$ is a semi-algebraic soliton.
\end{itemize}
\end{lemma}

\begin{lemma} If $\ngo_7(a,b,c,d,e)$  is the Lie algebra with Lie bracket given by
$$
[e_1,e_2]=-ae_4, \quad [e_1,e_7]=-b e_6, \quad[e_2,e_7]=-c e_5, [e_5,e_7]=-d e_3,\quad [e_6,e_7]=-e e_4,
$$
where $a,b,c,d,e\in \RR^*$ and $\vp_7:=e^{127}+e^{135}-e^{146}-e^{236}-e^{245}+e^{347}+e^{567}$, then
\begin{itemize}
\item[(i)] $\vp_7$ is closed if and only if $a=-b-c$ and $d=e$.
\item[(ii)]If $e^2 =\frac{b^2+c^2+bc}{2}$, then $(\ggo_7(b+c,b,c,e,e),\vp_7)$ is a semi-algebraic soliton.
\end{itemize}
\end{lemma}
The Lie bracket given in the above lemma is isomorphic to the one given by Table \ref{nl}, the isomorphism is given by:
$$
\left[\begin{smallmatrix}
0 & 1 & 0 & 0 & 0 & 0 & 0\\
0 & 0 & be/a & 0 & 0 & 0 & 0\\
0 & bcde/a & 0 & 0 & 0 & 0 &  bcde/a\\
0 & 0 & -be & 0 & 0 & be & 0\\
0 & 0 & 0 & 0 & - bce/a & 0 & 0\\
0 & 0 & 0 & -b & 0 & 0 & 0\\
1 & 0 & 0 & 0 & 0 & 0 & 0
\end{smallmatrix}\right].
$$
This concludes the proof of the theorem.
\end{proof}

The following tables provide information about the solitons found for the Lie algebras $\ngo_2,\dots,\ngo_7$. For any Lie algebra $\ngo_i$, the $3$-form
$\vp$, the number $\lambda$ and the derivation $D$ given in the tables are such that
$$
\Delta_\vp \vp=\lca_{X_D}\vp+\lambda \vp.
$$
Note that for each $i$, the given $\lambda$ is always positive. This implies that all the Laplacian solitons are expanding.

{\small \begin{table}
\renewcommand{\arraystretch}{1.6}
$$
\begin{array}{|c|c|c|}\hline
  & \ngo_2 & \ngo_3 \\
\hline\hline
\textbf{$\lb$} & [e_1,e_2]=-e_5,\hspace{0.1cm} [e_1,e_3]=-e_6.\hspace{0.1cm} & \begin{array}{l} [e_1,e_2]=-e_4,\hspace{0.1cm} [e_1,e_3]=(c-1)e_5,\;\ \\ \;\ [e_2,e_3]=-ce_6, \quad 0< c <   1/2. \end{array} \\
\hline
\textbf{$\vp$} & \begin{array}{l} e^{147}+e^{267}+e^{357} \;\ \\ \;\ +e^{123}+e^{156}+e^{245}-e^{346} \end{array} & \begin{array}{l} e^{123}+e^{145}+e^{167} \;\ \\ \;\ +e^{246}-e^{257}-e^{347}-e^{356} \end{array} \\
\hline \textbf{$\tau_\vp$} &  - e^{35}+ e^{26} &-ce^{16}+(1-c)e^{25}-e^{34} \\
\hline \textbf{$\Delta_\vp\vp$} & 2 e^{123} & 2 (1-c+c^2)  e^{123} \\
\hline\textbf{$\Ricci_\vp$}& -\Diag\left(1, \frac{1}{2}, \frac{1}{2}, 0, -\frac{1}{2}, -\frac{1}{2}, 0\right)&   \begin{array}{l}\frac{1}{2}\Diag(-2+2c-c^2, -1-c^2,\;\ \\ \;\  -1+2c-2c^2, 1, (-1+c)^2, c^2 , 0)\end{array}\\
\hline\textbf{$R_\vp$} & -1 &  -1+c-c^2\\
\hline \textbf{$\frac{R_\vp{^2}}{tr \Ricci_\vp^2}$} & \frac{1}{2} & \frac{1}{2}\\
\hline
\textbf{$\lambda$}  & 5 & 5 (1-c+c^2)   \\
\hline \textbf{$D$} & -\Diag(1,1,1,2,2,2,2) & -(1-c+c^2)\Diag(1,1,1,2,2,2,2) \\
\hline
\end{array}
$$
\caption{Laplacian solitons.}\label{n2-n3}
\end{table}}

{\small \begin{table}
\renewcommand{\arraystretch}{1.6}
$$
\begin{array}{|c|c|c|}\hline
     &   \ngo_4 & \ngo_5 \\
\hline\hline
\textbf{$\lb$} & \begin{array}{l}[e_1,e_2]=-\sqrt{2}e_3,\hspace{0.1cm}  [e_1,e_3]=-e_6,\;\ \\ \;\ \hspace{0.1cm}[e_2,e_4]=-\sqrt{2}e_6,\hspace{0.1cm} [e_1,e_5]=-e_7. \end{array} & \begin{array}{c}  [e_1,e_2]=-\sqrt{2}e_3, \hspace{0.1cm} [e_1,e_3]=-e_6, \;\ \\ \;\ \hspace{0.1cm} [e_1,e_4]=-e_7 , \hspace{0.1cm}[e_2,e_5]= -\sqrt{2}e_7 .\end{array} \\
\hline
\textbf{$\vp$} & \begin{array}{l} -e^{124}-e^{456}+e^{347} \;\ \\ \;\ +e^{135}+e^{167}+e^{257}-e^{236} \end{array} & \begin{array}{l} e^{134}+e^{457}-e^{246} \;\ \\ \;\ -e^{125}-e^{356}+e^{167}-e^{237} \end{array} \\
\hline \textbf{$\tau_\vp$} &-\sqrt{2}e^{34}+\sqrt{2}e^{16}-e^{56}+e^{37} &  - e^{46}+ e^{37}-\sqrt{2}e^{35}+\sqrt{2}e^{17}\\
\hline \textbf{$\Delta_\vp\vp$} & -4 e^{124} + 2 e^{135} + \sqrt{2} e^{245} + \sqrt{2} e^{127} & 2 e^{134} + \sqrt{2} e^{127} + \sqrt{2} e^{235} - 4 e^{125} \\
\hline\textbf{$\Ricci_\vp$} & \Diag\left(-2,-2, \frac{1}{2} ,-1 ,-\frac{1}{2} , \frac{3}{2} , \frac{1}{2} \right)  &   \Diag\left(-2, -2,\frac{1}{2},-\frac{1}{2},-1,\frac{1}{2},\frac{3}{2}\right)\\
\hline\textbf{$R_\vp$} & -3&  -3\\
\hline \textbf{$\frac{R_\vp{^2}}{tr \Ricci_\vp^2}$} & \frac{3}{4} & \frac{3}{4} \\
\hline
\textbf{$\lambda$}   & 9 & 9  \\
\hline \textbf{$D$}
                    &\left[\begin{smallmatrix}-1 & 0  & 0 & 0  & 0 & 0 & 0 \\0  & -2 & 0 & 0  & 0 & 0 & 0\\0  & 0  & -3 & 0  & 0 & 0 & 0\\0  & 0& 0 & -2 & 0 & 0 & 0\\ -\sqrt{2} & 0  &0& 0  & -3 & 0 & 0\\0  & 0  & 0 & 0  & 0 & -4 & 0\\0  & 0  & 0&-\sqrt{2} & 0 & 0 & -4\\\end{smallmatrix} \right]
                     &\left[\begin{smallmatrix}-1 & 0  & 0 & 0  & 0 & 0 & 0 \\0  & -2 & 0 & 0  & 0 & 0 & 0\\\sqrt{2}  & 0  & -3 & 0  & 0 & 0 & 0\\0  & 0  & 0 & -3 & 0 & 0 & 0 \\0  & 0  & 0 & 0 & -2 & 0 & \\0  & 0  & 0 & 0  & 0 & -4 & 0\\ 0  & 0  & 0 & 0 & -\sqrt{2} &  0 & -4\\\end{smallmatrix} \right]
                    \\
\hline
\end{array}
$$
\caption{Laplacian solitons.} \label{n4-n5}
\end{table}}

{\small \begin{table}
\renewcommand{\arraystretch}{1.6}
$$
\begin{array}{|c|c|c|}\hline
   &\ngo_6 & \ngo_7 \\
\hline\hline
\textbf{$\lb$}  & \begin{array}{c} [e_1,e_2]=-\sqrt{2}e_4,\hspace{0.1cm} [e_1,e_3]=-\sqrt{2}e_5 \;\ \\ \;\  [e_1,e_4]=-e_6,\hspace{0.1cm} [e_1,e_5]=-e_7. \end{array} &   \begin{array}{c} [e_1,e_2]=4e_4, \hspace{0.1cm} [e_1,e_7]=-2e_6, \;\ \\ \;\ [e_2,e_7]=-2e_5,\hspace{0.1cm} [e_5,e_7]=-\sqrt{6}e_5, \;\ \\ \;\ \hspace{0.1cm}[e_6,e_7]=-\sqrt{6}e_4 . \end{array}  \\
\hline
\textbf{$\vp$} & \begin{array}{c} e^{123}+e^{347}+e^{356}\;\ \\ \;\ +  e^{145}-e^{246}+e^{167}+e^{257} \end{array} & \begin{array}{c} e^{127}+e^{135}-e^{146} \;\ \\ \;\ -e^{236}-e^{245}+e^{347}+e^{567} \end{array} \\
\hline \textbf{$\tau_\vp$}   &  -\sqrt{2}e^{34}+\sqrt{2}e^{25}- e^{56}+ e^{47} & -2e^{15}+2e^{26}- \sqrt{6}e^{36}+ \sqrt{6}e^{45}-4e^{47} \\
\hline \textbf{$\Delta_\vp\vp$} & 4 e^{123} - \sqrt{2} e^{136} + \sqrt{2} e^{127} + 2 e^{145} &  \begin{array}{c}  24e^{127}-4\sqrt{6} e^{125} - 2\sqrt{6} e^{137} \;\ \\ \;\  +2\sqrt{6}e^{247} +12 e^{567} \end{array}\\
\hline \textbf{$\Ricci_\vp$} &  \Diag\left(-3, -1, -1, \frac{1}{2}, \frac{1}{2}, \frac{1}{2} , \frac{1}{2}\right)& \Diag\left(-10,-10,3,11,-1,-1,-10\right)\\
\hline\textbf{$R_\vp$} & -3&  -18\\
\hline \textbf{$\frac{R_\vp{^2}}{tr \Ricci_\vp^2}$} & \frac{3}{4} & \frac{3}{4} \\
\hline
\textbf{$\lambda$}   & 9 & 54 \\
\hline \textbf{$D$}
                   &\left[\begin{smallmatrix}-1 & 0  & 0 & 0  & 0 & 0 & 0 \\0  & -2 & 0 & 0  & 0 & 0 & 0\\0  & 0  & -2 & 0  & 0 & 0 & 0\\0  & 0  & 0 & -3 & 0 & 0 & 0 \\0  & 0  & 0 & 0 & -3 & 0 & 0\\0  & \sqrt{2}  & 0 & 0  & 0 & -4 & 0\\0  & 0  & \sqrt{2} & 0 & 0  &  0 & -4\\\end{smallmatrix} \right]
                    &\left[\begin{smallmatrix}-12 & 0  & 0 & 0  & 0 & 0 & 0 \\0  & -12  & 0 & 0  & 0 & 0 & 0 \\0  & -2\sqrt{6}  &-24  & 0  & 0 & 0 & 0 \\ -2\sqrt{6}  & 0  & 0 & -24   & 0 & 0 & 0 \\ 0  & 0  & 0 & 0  & -18 & 0 & -4\sqrt{6} \\ 0  & 0  & 0 & 0  & 0 & -18 & 0 \\ 0  & 0  & 0 & 0  & 0 & 0 & -6\\ \end{smallmatrix} \right] \\
\hline
\end{array}
$$
\caption{Laplacian solitons.}\label{n6-n7}
\end{table}}

Using  \eqref{lot-wei}, we computed $Q_\vp$ for any $\ngo_i$ with $i=2,\dots,7$:

$$Q_{\vp_2} = \frac{1}{3}\Diag(-2,-2,-2,1,1,1,1), \qquad Q_{\vp_3} = \frac{1-c+c^2}{3} \Diag(-2,-2,-2,1,1,1,1),$$
$$Q_{\vp_4} = \left[\begin{smallmatrix}-2 & 0  & 0 & 0 & \tfrac{1}{\sqrt{2}} & 0 & 0 \\ 0  & -1 & 0 & 0  & 0 & 0 & 0\\0  & 0 & 0 & 0  & 0 & 0 & 0\\0  & 0& 0& -1 & 0 & 0 & \tfrac{1}{\sqrt{2}}\\ \tfrac{1}{\sqrt{2}} & 0  &0& 0  & 0 & 0 & 0\\ 0 & 0  & 0 & 0  & 0 & 1 & 0\\0  & 0  & 0&\tfrac{1}{\sqrt{2}} & 0 & 0&1\\\end{smallmatrix}\right], \qquad Q_{\vp_5} = \left[\begin{smallmatrix}-2 & 0  &-\tfrac{1}{\sqrt{2}} & 0  & 0 & 0 & 0 \\0  & -1 & 0 & 0  & 0 & 0 &0\\-\tfrac{1}{\sqrt{2}} & 0 & 0 & 0 & 0 &0 & 0\\0  &0  & 0 & 0 &0 & 0 & 0 \\0  & 0  & 0 & 0 & -1 & 0 & \tfrac{1}{\sqrt{2}}\\0  & 0 &0 & 0  & 0 & 1 & 0\\ 0  & 0  & 0 & 0 & \tfrac{1}{\sqrt{2}} & 0&1\\ \end{smallmatrix}
\right],$$
$$Q_{\vp_6} = \left[\begin{smallmatrix}-2 & 0  & 0 & 0  &  -\frac{1}{\sqrt{2}} & 0 & 0 \\0  & -1 & 0 & 0  & 0 & 0 & 0\\0 & 0 & -1 & 0 &0 & 0 &  -\frac{1}{\sqrt{2}}\\0  & 0  & 0 & 0 & 0 & 0 & 0 \\0  & 0  & 0 & 0 & 0 & 0 & 0\\0  &  -\frac{1}{\sqrt{2}}  & 0 & 0  & 0 & 1 & 0\\0  & 0  &  -\frac{1}{\sqrt{2}} & 0 & 0  & 0 &1\\\end{smallmatrix} \right], \qquad Q_{\vp_7} = \left[\begin{smallmatrix} -4  & 0  & 0 & 0  & -1 & 0 & 0 \\0  &  -4  & 0 & 0  & 0 & 1 & 0 \\0  & 0  & 9  & 0  & 0 & -\sqrt{6}/2 & 0 \\  0  & 0  & 0 & 17  & \sqrt{6}/2 & 0 & -2 \\ 1  & 0  & 0 & -\sqrt{6}/2  & 5 & 0 & 0 \\ 0  & -1  & \sqrt{6}/2 & 0  & 0 & 5 & 0 \\ 0  & 0  & 0 & 2  & 0& 0 & -4 \\ \end{smallmatrix}
\right].$$

Note that in every case, we have that $Q_\vp=-\frac{1}{3}\lambda I -D$, where $\lambda$ and $D$ are given in the tables above.

\proof[Acknowledgements]  I would like to express my sincere gratitude to my Ph.D. advisor Prof. Jorge Lauret for the continuous support during the writing of this paper, for his patience and motivation.

\end{document}